\documentclass[reqno,12pt]{amsart}
\usepackage[latin1]{inputenc}
\usepackage[english]{babel}
\usepackage{amsmath, amsthm, amssymb, amsopn, amsfonts, amstext, stmaryrd, enumerate, color, mathrsfs, enumitem, url, accents, hyperref}
\usepackage[final]{showkeys}

\usepackage[margin=1.2in]{geometry}
\hypersetup{
	colorlinks=true,
	linkcolor=blue,
	citecolor=red,
	urlcolor=black
}

\newcommand{\C}{\mathscr{C}}

\renewcommand{\L}{\mathcal{L}}
\newcommand{\N}{\mathbb{N}}
\newcommand{\R}{\mathbb{R}}

\newcommand{\loc}{{\rm loc}}

\newcommand{\PV}{\mbox{\normalfont P.V.}}
\newcommand{\Haus}{\mathcal{H}}
\newcommand{\red}{{\partial^*}}

\DeclareMathOperator{\Per}{Per}

\newcommand{\cu}{\mathscr H}

\def\XXint#1#2#3{{\setbox0=\hbox{$#1{#2#3}{\int}$ }
\vcenter{\hbox{$#2#3$ }}\kern-.6\wd0}}

\newlength{\dhatheight}

\numberwithin{equation}{section}

\theoremstyle{plain}
\newtheorem{definition}{Definition}[section]
\newtheorem{theorem}[definition]{Theorem}
\newtheorem{proposition}[definition]{Proposition}
\newtheorem{lemma}[definition]{Lemma}
\newtheorem{corollary}[definition]{Corollary}

\theoremstyle{definition}
\newtheorem{remark}[definition]{Remark}

\renewcommand{\le}{\leqslant}

\renewcommand{\ge}{\geqslant}

\title[Bernstein-Moser-type results for nonlocal minimal graphs]{Bernstein-Moser-type results\\for nonlocal minimal graphs}

\author{Matteo Cozzi}
\author{Alberto Farina}
\author{Luca Lombardini}

\address{\vspace{-15pt}
\newline
\textit{Matteo Cozzi}
\newline
University of Bath, Department of Mathematical Sciences, Bath BA2 7AY, UK
\newline
\textit{E-mail address}: \textit{\tt m.cozzi@bath.ac.uk}
}

\address{\vspace{-13pt}
\newline
\textit{Alberto Farina}
\newline
Universit\'e de Picardie ``Jules Verne'', Laboratoire Ami\'enois de Math\'ematique Fondamentale et Appliqu\'ee, CNRS UMR 7352, 33 Rue St Leu, 80039 Amiens, France
\newline
\textit{E-mail address}: \textit{\tt alberto.farina@u-picardie.fr}
}

\address{\vspace{-13pt}
\newline
\textit{Luca Lombardini}
\newline
Universit\`a degli Studi di Milano, Dipartimento di Matematica ``Federigo Enriques'', Via Saldini 50, 20133 Milano, Italy
\newline
Universit\'e de Picardie ``Jules Verne'', Laboratoire Ami\'enois de Math\'ematique Fondamentale et Appliqu\'ee, CNRS UMR 7352, 33 Rue St Leu, 80039 Amiens, France
\newline
University of Western Australia, Department of Mathematics and Statistics, 35 Stirling Highway, WA 6009 Crawley, Australia
\newline
\textit{E-mail address}: \textit{\tt luca.lombardini@unimi.it}
}



\thanks{The first author acknowledges support from a Royal Society Newton International Fellowship, from the MINECO grants MTM2014-52402-C3-1-P and MTM2017-84214-C2-1-P, and from the Mar\'ia de Maeztu Programme for Units of Excellence in R\&D with project code MDM-2014-0445. Part of this work has been carried out while the first and second authors were visiting the Universit\`a degli Studi di Milano, which they thank for the warm hospitality.}

\begin{document}

\begin{abstract}
We prove a flatness result for entire nonlocal minimal graphs having some partial derivatives bounded from either above or below. This result generalizes fractional versions of classical theorems due to Bernstein and Moser.
%Furthermore, it provides a partial extension of a sharp rigidity result recently obtained by the second author for standard minimal graphs.
Our arguments rely on a general splitting result for blow-downs of nonlocal minimal graphs.\\
Employing similar ideas, we establish that entire nonlocal minimal graphs bounded on one side by a cone are affine.\\
Moreover, we show that entire graphs having constant nonlocal mean curvature are minimal, thus extending a celebrated result of Chern on classical CMC graphs.
\end{abstract}

\maketitle

\section{Introduction and main results}

\noindent
Let~$n \ge 1$ be an integer and~$\alpha \in (0, 1)$. Given an open set~$\Omega \subseteq \R^{n + 1}$ and a measurable set~$E \subseteq \R^{n + 1}$, we define the~\emph{$\alpha$-perimeter} of~$E$ in~$\Omega$ by
$$
\Per_\alpha(E, \Omega) := \int_{\Omega \cap E} \int_{\R^{n + 1} \setminus E} \frac{dx dy}{|x - y|^{n + 1 + \alpha}} + \int_{E \setminus \Omega} \int_{\Omega \setminus E} \frac{dx dy}{|x - y|^{n + 1 + \alpha}}.
$$
A measurable set~$E \subseteq \R^{n + 1}$ is called~\emph{$\alpha$-minimal} in~$\Omega$ if it satisfies~$\Per_\alpha(E, \Omega) < +\infty$ and~$\Per_\alpha(E, \Omega) \le \Per_\alpha(F, \Omega)$ for every~$F \subseteq \R^{n + 1}$ such that~$F \setminus \Omega = E \setminus \Omega$. Sets that minimize~$\Per_\alpha$ in all bounded open subsets of~$\R^{n + 1}$ will be simply called~$\alpha$-minimal and their boundaries~\emph{$\alpha$-minimal surfaces}.

Fractional (or nonlocal) perimeters and their minimizers have been first introduced by Caffarelli, Roquejoffre \& Savin~\cite{CRS10} in~2010, motivated by applications to phase transition problems in the presence of long range interactions. There, the authors established several results about~$\alpha$-minimal surfaces, concerning in particular their existence and regularity. They also showed that every minimizer~$E$ of~$\Per_\alpha$ satisfies the Euler-Lagrange equation
$$
H_\alpha[E](x) = 0 \quad \mbox{for } x \in \partial E
$$
in a suitable viscosity sense. The quantity~$H_\alpha[E](x)$ is often referred to as the~\emph{$\alpha$-mean curvature} of~$E$ at~$x \in \partial E$ and is formally defined by
\begin{equation} \label{Halphadef}
H_\alpha[E](x) := \PV \int_{\R^{n+1}} \frac{\chi_{\R^{n+1}\setminus E}(y) - \chi_{E}(y)}{|x - y|^{n + 1 + \alpha}} \, dy.
\end{equation}
In the subsequent years, many authors have directed their attention towards~$\alpha$-minimal surfaces, obtaining a variety of results mostly regarding their regularity and qualitative behavior. We encourage the reader to consult the surveys contained in~\cite{V13},~\cite[Chapter~6]{BV16},~\cite{DV18}, and~\cite[Section~7]{CF17} for more information.

In this brief note we are mostly interested in~$\alpha$-minimal sets~$E \subseteq \R^{n + 1}$ that are subgraphs of a measurable function~$u: \R^n \to \R$, i.e., that satisfy
\begin{equation} \label{Esubgraph}
E = \left\{ x = (x', x_{n + 1}) \in \R^n \times \R : x_{n + 1} < u(x') \right\}.
\end{equation}
We will call the boundaries of such extremal sets~\emph{$\alpha$-minimal graphs}.

Note that, when~$E$ is the subgraph of a function~$u$, we can write its~$\alpha$-mean curvature as an integrodifferential operator acting on~$u$. More precisely, letting~$u: \R^n \to \R$ be a function of, say, class~$C^{1, 1}$ in a neighborhood of a point~$x' \in \R^n$ and~$E$ be given by~\eqref{Esubgraph}, we have that
\begin{equation} \label{HE=Hu}
H_\alpha[E](x',u(x')) = \cu_\alpha u(x'),
\end{equation}
with
\begin{equation} \label{Hcorsdef}
\cu_\alpha u(x') := 2 \, \PV \int_{\R^n} G \left( \frac{u(x')-u(y')}{|x' - y'|} \right) \frac{dy'}{|x' - y'|^{n+\alpha}} 
\end{equation}
and
\begin{equation} \label{Gdef}
G(t):=\int_0^t\frac{d\tau}{(1+\tau^2)^\frac{n+1+\alpha}{2}} \quad \mbox{for } t\in\R.
\end{equation}
Both here and in~\eqref{Halphadef} the symbol~$\PV$ means that the integrals must be understood in the Cauchy principal value sense. See, e.g.,~\cite[Section~2]{CV13} or~\cite[Appendix~B]{BLV16} for a proof of identity~\eqref{HE=Hu}.
%It can be easily seen that~$\cu_\alpha u(x')$ is well defined, thanks to the regularity of~$u$ around~$x'$.

Taking advantage of the convexity of the energy functional associated to~$\cu_\alpha$ and of a suitable rearrangement inequality, it will be shown in~\cite{CL18} that a set~$E$ given by~\eqref{Esubgraph} for some function~$u: \R^n \to \R$ is~$\alpha$-minimal if and only if~$u$ is a solution of
\begin{equation} \label{Hu=0}
\cu_\alpha u = 0 \quad \mbox{in } \R^n.
\end{equation}
There are several notions of solutions of~\eqref{Hu=0}, such as smooth solutions, viscosity solutions, and weak solutions. However, all such definitions are equivalent under mild assumptions on~$u$---for more details, see the forthcoming~\cite{CL18} or Chapter~4 of the PhD thesis~\cite{Lthesis} of the third author (and in particular~\cite[Corollary~4.1.12]{Lthesis}). In what follows, a solution of~\eqref{Hu=0} will always be a function~$u \in C^\infty(\R^n)$ that satisfies identity~\eqref{Hu=0} pointwise. We stress that no growth assumptions at infinity are made on~$u$.

The main contribution of this note is the following result.

\begin{theorem} \label{underPakmainthm}
Let~$n \ge \ell \ge 1$ be integers,~$\alpha \in (0, 1)$, and suppose that
\begin{equation*} \tag{$P_{\alpha, \ell}$} \label{Paellprop}
\mbox{there exist no singular~$\alpha$-minimal cones in~$\R^\ell$.}
\end{equation*}
Let~$u$ be a solution of~\eqref{Hu=0} having~$n - \ell$ partial derivatives bounded on one side.

Then,~$u$ is an affine function.
\end{theorem}

%Condition~\eqref{Paellprop} is crucial for the validity of Theorem~\ref{underPakmainthm}.
We point out that throughout the paper a \emph{cone} is any subset~$\C$ of the Euclidean space for which~$\lambda x \in \C$ for every~$x \in \C$ and~$\lambda > 0$. A set~$E$ will be said to be \emph{trivial} if either~$E$ or its complement has measure zero. In addition, a \emph{singular} cone is a cone whose boundary is not smooth at the origin or, equivalently, any nontrivial cone that is not a half-space.

Characterizing the values of~$\alpha$ and~$\ell$ for which~\eqref{Paellprop} is satisfied represents a challenging open problem, whose solution would lead to fundamental advances in the understanding of the regularity properties enjoyed by nonlocal minimal surfaces. Currently, property~\eqref{Paellprop} is know to hold in the following cases:
\begin{itemize}[leftmargin=*]
\item when~$\ell = 1$ or~$\ell = 2$, for every~$\alpha \in (0, 1)$;
\item when~$3\le \ell \le 7$ and~$\alpha \in (1 - \varepsilon_0, 1)$ for some small~$\varepsilon_0 \in (0,1]$ depending only on~$\ell$. 
\end{itemize}
Case~$\ell = 1$ holds by definition, while~$\ell = 2$ is the content of~\cite[Theorem~1]{SV13}. On the other hand, case~$3 \le \ell \le 7$ has been established in~\cite[Theorem~2]{CV13}---see also~\cite{CCS17} for a different approach yielding an explicit value for~$\varepsilon_0$ when~$\ell = 3$.

As a consequence of Theorem~\ref{underPakmainthm} and the last remarks, we immediately obtain the following result.

\begin{corollary} \label{mainthm}
Let~$n \ge \ell \ge 1$ be integers and~$\alpha \in (0, 1)$. Assume that either
\begin{itemize}[leftmargin=*]
\item $\ell \in \{ 1, 2 \}$, or
\item $3\le \ell \le 7$ and~$\alpha \in (1 - \varepsilon_0, 1)$, with~$\varepsilon_0=\varepsilon_0(\ell) > 0$ as in~\cite[Theorem~2]{CV13}.
\end{itemize}
Let~$u$ be a solution of~\eqref{Hu=0} having~$n-\ell$ partial derivatives bounded on one side.

Then,~$u$ is an affine function.
\end{corollary}

%Corollary~\ref{mainthm} generalizes several known rigidity results for~$\alpha$-minimal graphs. Taking~$\ell = n$ we recover the fractional Bernstein theorems of~\cite{FV17}.
We observe that
Theorem~\ref{underPakmainthm} gives a new flatness result for~$\alpha$-minimal graphs, under the assumption that~\eqref{Paellprop} holds true. It can be seen as a generalization of the fractional~De~Giorgi-type lemma contained in~\cite[Theorem~1.2]{FV17}, which is recovered here taking~$\ell = n$. In this case, we indeed provide an alternative proof of the result of~\cite{FV17}.

On the other hand, the choice~$\ell = 2$ gives an improvement of~\cite[Theorem~4]{FarV17}, when specialized to~$\alpha$-minimal graphs. In light of these observations, Theorem~\ref{underPakmainthm} and Corollary~\ref{mainthm} can be seen as a bridge between Bernstein-type theorems (flatness results in low dimensions) and Moser-type theorems (flatness results under global gradient bounds).

For classical minimal graphs---formally corresponding to the case~$\alpha = 1$ here (see, e.g.,~\cite{ADM11,CV13})---the counterpart of Corollary~\ref{mainthm} has been recently obtained by the second author in~\cite{F17}. In that case, the result is sharp and holds with~$\ell = \min \{ n, 7 \}$.
%This is essentially a consequence of the fact that property~\eqref{Paellprop} for~$\alpha = 1$ is valid if and only if~$1 \le \ell \le 7$.
See also~\cite{F15} by the same author for a previous result established for~$\ell = 1$ and through a different argument.

Using the same ideas that lead to Theorem~\ref{underPakmainthm}, we can prove the following rigidity result for entire~$\alpha$-minimal graphs that lie above a cone.

%We also mention~\cite[Theorem~1.5]{CC17}, which is a flatness result for~$\alpha$-minimal graphs that grow at most linearly at infinity. By exploiting the same ideas used in the proof of Theorem~\ref{underPakmainthm}, we can prove the following stronger result, which requires a linear control on the growth at infinity only from below---thus extending [Theorem BOH] to the fractional and nonlocal framework.

\begin{theorem}\label{growthTeo}
Let~$n\ge1$ be an integer and~$\alpha\in(0,1)$. Let~$u$ be a solution of~\eqref{Hu=0} and assume that there exists a constant~$C > 0$ for which
\begin{equation} \label{ugecone}
u(x')\ge - C (1+|x'|)\quad\mbox{for every }x'\in\R^n.
\end{equation}
Then,~$u$ is an affine function.
\end{theorem}

Of course, the same conclusion can be drawn if~\eqref{ugecone} is replaced by the specular
$$
u(x') \le C (1 + |x'|) \quad \mbox{for every } x' \in \R^n.
$$

For classical minimal graphs, the corresponding version of Theorem~\ref{growthTeo} follows at once from the gradient estimate of Bombieri, De Giorgi \& Miranda~\cite{BDM69} and Moser's version of Bernstein's theorem~\cite{M61}. See for instance~\cite[Theorem~17.6]{G84} for a clean statement and the details of its proof.

In the nonlocal scenario, a gradient bound for~$\alpha$-minimal graphs has been recently established in~\cite{CC17}. However, this result is partly weaker than the one of~\cite{BDM69}, since it provides a bound for the gradient of a solution of~\eqref{Hu=0} in terms of its oscillation, and not just of its supremum (or infimum) as in~\cite{BDM69}. Consequently, in~\cite{CC17} a rigidity result analogous to Theorem~\ref{growthTeo} is deduced, but with~\eqref{ugecone} replaced by the stronger, two-sided assumption:~$|u(x')| \le C(1 + |x'|)$ for every~$x' \in \R^n$. Theorem~\ref{growthTeo} thus improves~\cite[Theorem~1.6]{CC17} directly. Moreover, our proof is different, as it relies on geometric considerations rather than uniform regularity estimates.

The proof of Theorem~\ref{underPakmainthm} is based on the extension to the fractional framework of a strategy devised by the second author for classical minimal graphs and previously unpublished. As a result, the ideas contained in the following sections can be used to obtain a different, easier proof of~\cite[Theorem~1.1]{F17}---since, by Simons' theorem (see, e.g.,~\cite[Theorem~28.10]{Maggi}), no singular classical minimal cones exist in dimension lower or equal to~$7$. Similarly, the same argument that we employ for Theorem~\ref{growthTeo} can be successfully applied to classical minimal graphs, giving a different, more geometric, proof of~\cite[Theorem~17.6]{G84}.

The argument leading to Theorem~\ref{underPakmainthm} relies on a general splitting result for blow-downs of~$\alpha$-minimal graphs. Since it may have an interest on its own, we provide its statement here below.

\begin{theorem} \label{singblowdownthm}
Let~$n \ge 1$ be an integer and~$\alpha \in (0, 1)$. Let~$u$ be a solution of~\eqref{Hu=0} and~$E$ as in~\eqref{Esubgraph}. Assume that~$u$ is not affine and that, for some~$k \in \{ 1, \ldots, n - 1 \}$, the partial derivative~$\frac{\partial u}{\partial x_i}$ is bounded from below in~$\R^n$ for every~$i = 1, \ldots, k$.

Then, every blow-down limit~$\C \subseteq \R^{n + 1}$ of~$E$ is a cylinder of the form
$$
\C = \R^k \times P \times \R,
$$
for some singular~$\alpha$-minimal cone~$P \subseteq \R^{n - k}$.
\end{theorem}

The notion of blow-down limit will be made precise in Section~\ref{blowsec}.

\begin{remark}
As revealed by a simple inspection of its proof, Theorem~\ref{singblowdownthm} still holds if we require any~$k$ directional derivatives~$\partial_{\nu_1}u,\ldots,\partial_{\nu_k}u$ (not necessarily the partial derivatives) to be bounded from below, provided that the directions~$\nu_1,\ldots,\nu_k$ are linearly independent. Consequently, one can similarly modify the statements of Theorem~\ref{underPakmainthm} and Corollary~\ref{mainthm} without affecting their validity.
\end{remark}

Theorem~\ref{growthTeo} says in particular that there exist no non-flat~$\alpha$-minimal subgraphs that contain a half-space. Actually, one can prove the  
following theorem, valid not only for~$\alpha$-minimal subgraphs, but for general minimizers of the~$\alpha$-perimeter.

\begin{theorem} \label{nononflatthm}
Let~$n \ge 1$ be an integer and~$\alpha \in (0, 1)$. If~$E$ is a nontrivial~$\alpha$-minimal set in~$\R^{n + 1}$ that contains a half-space, then~$E$ is a half-space.
\end{theorem}

Theorem~\ref{nononflatthm} already appeared in the literature---see~\cite[Lemma~8.3]{DSV18}. For the reader's convenience, we nevertheless include a brief and slightly different proof of it in Section~\ref{nononflatsec}.

Interestingly, Theorem~\ref{nononflatthm} can be used to obtain a stronger version of Theorem~\ref{growthTeo}, where the bound in~\eqref{ugecone} is required to only hold at all points~$x'$ that lie in a half-space of~$\R^n$. See Remark~\ref{stressedrmk} at the end of Section~\ref{growthsec}.

The remainder of the paper is structured as follows. In Section~\ref{blowsec} we gather some known facts about sets with finite perimeter, the regularity of~$\alpha$-minimal surfaces, and their blow-downs. Section~\ref{splittingsec} is devoted to the proof of Theorem~\ref{singblowdownthm}, while in Section~\ref{mainsec} we show how Theorem~\ref{underPakmainthm} follows from it. Sections~\ref{nononflatsec} and~\ref{growthsec} contain the proofs of Theorems~\ref{nononflatthm} and~\ref{growthTeo}, respectively. The note is closed by Section~\ref{constapp}, which includes the extension of a result due to~Chern~\cite{C65} to the framework of graphs having constant~$\alpha$-mean curvature.

\section{Some remarks on nonlocal minimal surfaces and blow-down cones} \label{blowsec}

\noindent
As customary when dealing with the perimeter (either classical or fractional), we implicitly assume that all the sets we consider
contain their measure theoretic interior, do not intersect their measure theoretic exterior, and are such that their topological boundary coincides with their measure theoretic boundary---which is possible up to modifications in a set of Lebesgue measure zero.

More precisely, given a measurable set $E\subseteq\R^{n+1}$ we define
\begin{align*}
E_{\rm int} := & \hspace{3pt} \left\{ x \in \R^{n + 1} : |E \cap B_r(x)|=|B_1|r^{n+1} \mbox{ for some } r > 0 \right\},\\
E_{\rm ext} := & \hspace{3pt} \left\{x \in \R^{n + 1} : |E \cap B_r(x)|=0 \mbox{ for some } r > 0 \right\},
\end{align*}
and
\begin{align*}
\partial^-E := & \hspace{3pt} \R^{n+1}\setminus\big(E_{\rm int}\cup E_{\rm ext}\big) \\
= & \left\{ x \in \R^{n + 1} : 0<|E\cap B_r(x)|<|B_1|r^{n+1} \mbox{ for all } r > 0 \right\}.
\end{align*}
Then, we assume that
\[
E_{\rm int}\subseteq E,\quad E_{\rm ext}\cap E=\varnothing,\quad\mbox{and}\quad
\partial E=\partial^-E.
\]
See, e.g., step two in the proof of~\cite[Proposition~12.19]{Maggi} and Section~3.2 of~\cite{V91}. Notice that this requirement amounts to identifying the set~$E$ with a specific representative within its~$L^1_\loc$ class. Since
\[
\Per_\alpha(F,\Omega)=\Per_\alpha(E,\Omega) \quad \mbox{for every set } F \subseteq \R^{n + 1} \mbox{ such that } |E \Delta F| = 0,
\]
such an assumption does not affect the~$\alpha$-perimeter of~$E$.

We now recall some known results about the regularity of $\alpha$-minimal surfaces, which will be often used without mention in the subsequent sections.

Let~$E\subseteq\R^{n+1}$ be an~$\alpha$-minimal set. Then, its boundary~$\partial E$ is~$n$-rectifiable. Actually, by~\cite[Theorem~2.4]{CRS10},~\cite[Corollary~2]{SV13}, and~\cite[Theorem~1.1]{FV17},~$\partial E$ is locally of class~$C^\infty$, except possibly for a set of singular points~$\Sigma_E\subseteq\partial E$ satisfying
\[
\mathcal H^d(\Sigma_E)=0 \quad \mbox{for every }d>n-2.
\]
In particular, the set~$E$ has locally finite (classical) perimeter in~$\R^{n+1}$ and actually, as proved in~\cite{CSV18}, uniform perimeter estimates are available. Thus, it makes sense to consider its reduced boundary~$\red \! E$.

Furthermore, thanks to the blow-up analysis developed in~\cite{CRS10}---see in particular~\cite[Theorem~9.4]{CRS10}---and the tangential properties of the reduced boundary of a set of locally finite perimeter---see, e.g., \cite[Theorem~15.5]{Maggi}---we have that~$\red \! E$ is smooth and the singular set is given by
\[
\Sigma_E=\partial E\setminus\red \!E.
\]

Given a measurable set~$E \subseteq \R^{n + 1}$, a point~$x \in \R^{n + 1}$, and a real number~$r > 0$, we write
$$
E_{x, r} := \frac{E - x}{r}.
$$
We call any~$L^1_\loc$-limit~$E_{x, \infty}$ of~$E_{x, r_j}$ along a diverging sequence~$\{ r_j \}$ a \emph{blow-down limit} of~$E$ at~$x$.

Observe that doing a blow-down of a set~$E$ corresponds to the operation of looking at~$E$ from further and further away. As a result, in the limit one loses track of the point at which the blow-down was centered. That is, blow-down limits may depend on the chosen diverging sequence~$\{ r_j \}$ but not on the point of application~$x$. This fact is certainly well-known to the experts. Nevertheless, we include in the following Remark a brief justification of it for the convenience of the less experienced reader.

\begin{remark}\label{tang_cone_infty}
Let~$x, y \in \R^{n + 1}$ and~$E \subseteq \R^{n + 1}$ be a measurable set. Assume that there exists a set~$F \subseteq \R^{n + 1}$ such that~$E_{x, r_j} \rightarrow F$ in~$L^1_\loc(\R^{n + 1})$ as~$j \rightarrow +\infty$, along a diverging sequence~$\{ r_j \}$. We claim that also
\begin{equation} \label{E2conv}
E_{y, r_j} \rightarrow F \mbox{ in } L^1_\loc(\R^{n+1}) \mbox{ as } j \rightarrow +\infty.
\end{equation}
To verify this assertion, let~$R > 0$ be fixed and write~$f_j := \chi_{E_{x, r_j}}$ and~$f := \chi_{F}$. Notice that~$\chi_{E_{y, r_j}} = \tau_{v_j} f_j := f_j(\cdot - v_j)$, with~$v_j := (x - y) / r_j$. Since~$v_j \rightarrow 0$ as~$j \rightarrow 0$, we have
\begin{align*}
\left| (E_{y, r_j} \Delta F) \cap B_R \right| & = \| \chi_{E_{y, r_j}} - \chi_F \|_{L^1(B_R)} = \| \tau_{v_j} f_j - f \|_{L^1(B_R)} \\
& \le \| \tau_{v_j} f_j - \tau_{v_j} f \|_{L^1(B_R)} + \| \tau_{v_j} f - f \|_{L^1(B_R)}\\
& \le  \| f_j - f \|_{L^1(B_{R + 1})} + \| \tau_{v_j} f - f \|_{L^1(B_R)},
\end{align*}
provided~$j$ is sufficiently large. Claim~\eqref{E2conv} follows since, by assumption,~$f_j \rightarrow f$ in~$L^1_\loc(\R^{n + 1})$ and~$R > 0$ is arbitrary.
\end{remark}

In light of this remark, we can assume blow-downs to be always centered at the origin. For simplicity of notation, we will write~$E_r := E_{0, r} = E/r$ and use~$E_\infty$ to indicate any blow-down limit.

The next lemma collects some known facts about blow-downs of~$\alpha$-minimal sets.

\begin{lemma} \label{blowdownlem}
Let~$E \subseteq \R^{n + 1}$ be a nontrivial~$\alpha$-minimal set. Then, for every diverging sequence~$\{ r_j \}$, there exists a subsequence~$\{ r_{j_k} \}$ of~$\{ r_j \}$ and a set~$E_\infty \subseteq \R^{n + 1}$ such that~$E_{r_{j_k}} \rightarrow E_\infty$ in~$L^1_\loc(\R^{n + 1})$ as~$k \rightarrow +\infty$. The set~$E_\infty$ is a nontrivial~$\alpha$-minimal cone. Furthermore,~$E_\infty$ is a half-space if and only if~$E$ is a half-space.
\end{lemma}
\begin{proof}
The existence of a limit of~$E_{r_j}$ (up to a subsequence) is a consequence of uniform estimates for the~$\alpha$-perimeter of~$\alpha$-minimal sets and the compactness of the fractional Sobolev embedding. More in detail, by the scale invariance of~$\Per_\alpha$, we have that~$E_{r}$ is an~$\alpha$-minimal set. Hence, for every~$r, R > 0$,
$$
\Per_\alpha(E_{r}, B_R) \le \Per_\alpha(E_{r} \setminus B_R, B_R) \le \int_{B_R} \int_{\R^{n + 1} \setminus B_R} \frac{dx dy}{|x - y|^{n + 1 + \alpha}} \le C R^{n+1 - \alpha},
$$
for some constant~$C > 0$ depending only on~$n$ and~$\alpha$. In particular, for fixed~$R > 0$, the quantities~$[\chi_{E_{r_j}}]_{W^{1, \alpha}(B_R)}$ are bounded uniformly in~$j \in \N$. By, say,~\cite[Theorem~7.1]{DPV12}, there exist therefore a subsequence~$\{ r_{j_k}^{(R)} \}$ and a set~$E_\infty^{(R)} \subseteq B_R$ for which~$E_{r_{j_k}^{(R)}} \rightarrow E_\infty^{(R)}$ in~$L^1(B_R)$ as~$k\rightarrow +\infty$. A standard diagonal argument then yields the existence of a limit~$E_\infty \subseteq \R^{n + 1}$ in~$L^1_\loc(\R^{n + 1})$ along some subsequence~$\{ r_{j_k} \}$.

The fact that~$E_\infty$ is~$\alpha$-minimal is a consequence of the~$\alpha$-minimality of the sets~$E_{r_{j_k}}$
and their~$L^1_\loc$ convergence to~$E_\infty$---see~\cite[Theorem~3.3]{CRS10}.

Next we observe that, since~$E$ is nontrivial, we can find a point~$x\in\partial E$. Thanks to
Remark~\ref{tang_cone_infty}, we then have that
$$
E_{x,r_{j_k}}\to E_\infty\mbox{ in }L^1_\loc(\R^{n+1})\mbox{ as }k\to\infty.
$$
Since~$0\in\partial E_{x,r_{j_k}}$ for every~$k\in\mathbb N$,
we can conclude that~$E_\infty$ is a cone by arguing as in~\cite[Theorem~9.2]{CRS10}.

The nontriviality of~$E_\infty$ can be established, for instance, by using the uniform density estimates of~\cite{CRS10}. Indeed,~$0 \in \partial E_{x,r_{j_k}}$ for every~$k \in \N$ and hence~\cite[Theorem~4.1]{CRS10} gives that~$\min \{ |E_{x, r_{j_k}} \cap B_1|, |B_1 \setminus E_{x, r_{j_k}}| \} \ge c$ for some constant~$c > 0$ independent of~$k$. As~$E_{x,r_{j_k}} \rightarrow E_\infty$ in~$L^1(B_1)$, it follows that both~$E_\infty$ and its complement have positive measure in~$B_1$. Consequently,~$E_\infty$ is neither the empty set nor the whole~$\R^{n + 1}$.

Finally, if~$E_\infty$ is a half-space, one can deduce the flatness of~$\partial E$ from the~$\varepsilon$-regularity theory of~\cite[Section~6]{CRS10} and the fact that~$\partial E_{r_{j_k}} \rightarrow \partial E_\infty$ in the Hausdorff sense, thanks to the uniform density estimates. See, e.g.,~\cite[Lemma~3.1]{FV17} for more details on this argument.
\end{proof}

\section{Proof of Theorem~\ref{singblowdownthm}} \label{splittingsec}

\noindent
In this section we include a proof of the splitting result stated in the introduction, namely Theorem~\ref{singblowdownthm}. The argument leading to it is based on the following classification result for nonlocal minimal cones that contain their translates. For classical minimal cones, it was proved in~\cite{GMM97}.

\begin{proposition} \label{GMMprop}
Let~$\C \subseteq \R^{n + 1}$ be an~$\alpha$-minimal cone and assume that
\begin{equation} \label{C+vinC}
\C + v \subseteq \C
\end{equation}
for some~$v \in \R^{n + 1} \setminus \{ 0 \}$. Then,~$\C$ is either a half-space or a cylinder in direction~$v$.
\end{proposition}
\begin{proof}
First of all, we notice that, since~$\C$ is a cone and inclusion~\eqref{C+vinC} holds true, the function~$w := - \nu_\C \cdot v$ satisfies
\begin{equation} \label{wge0}
w \ge 0 \quad \mbox{in } \red \C.
\end{equation}
To see this, let~$x \in \red \C$ and observe that,~$\C$ being a cone, we have that~$\mu x \in \overline{\C}$ for every~$\mu > 0$. But then~$\mu x + v \in \overline{\C} + v$ and, using~\eqref{C+vinC}, it follows that~$\mu x + v \in \overline{\C}$. Consequently,~$\mu \lambda x + \lambda v = \lambda (\mu x + v) \in \overline{\C}$ for every~$\lambda, \mu > 0$. Choosing~$\mu = 1/\lambda$ we get that~$x + \lambda v \in \overline{\C}$ for every~$\lambda > 0$, which gives that~$v$ points inside~$\overline{\C}$. Recalling that the normal~$\nu_\C$ points outside~$\C$, we are immediately led to~\eqref{wge0}.

Now, by~\cite[Theorem~1.3(i)]{CC17} we know that~$w$ solves
\begin{equation} \label{WJharm}
\L w + c^2 w = 0 \quad \mbox{in } \red \C,
\end{equation}
where
\begin{align*}
\L w(x) & := \PV \int_{\red \C} \frac{w(y) - w(x)}{|x - y|^{n + 1 + \alpha}} \, d\Haus^n(y),\\
c^2(x) & := \frac{1}{2} \int_{\red \C} \frac{|\nu_\C(x) - \nu_\C(y)|^2}{|x - y|^{n + 1 + \alpha}} \, d\Haus^n(y),
\end{align*}
for every~$x \in \red \C$. As~$c^2 \ge 0$ in~$\red C$ and~\eqref{wge0} holds true, we deduce from~\eqref{WJharm} that~$w$ is~$\L$-superharmonic in~$\red \C$, i.e.,
$$
- \L w \ge 0 \quad \mbox{in } \red \C.
$$
By~\cite[Corollary~6.9]{CC17} (and the perimeter estimate of~\cite{CSV18}), we then infer that, for every point~$x \in \red \C$ and radius~$R > 0$, the function~$w$ satisfies
$$
\inf_{B_R(x) \cap \red \C} w \ge c_\star R^{1 + \alpha} \int_{\red \C} \frac{w(y)}{(R + |y - x|)^{n + 1 + \alpha}} \, d\Haus^n(y),
$$
for some constant~$c_\star \in (0, 1]$ depending only on~$n$ and~$\alpha$.

Accordingly, either~$w = 0$ in the whole~$\red \C$ or~$\inf_{B_R(x) \cap \red \C} w \ge c_{x, R}$ for some constant~$c_{x, R} > 0$ and for every~$x \in \red \C$ and~$R > 0$. In the first case, it is easy to see that~$\C$ must be a cylinder in direction~$v$. If the second situation occurs, then~$\partial \C$ is a locally Lipschitz graph with respect to the direction~$v$ (see, e.g.,~\cite[Theorem~5.6]{M64}), and hence smooth, due to~\cite[Theorem~1.1]{FV17}. It being a cone, we conclude that~$\C$ must be a half-space.
\end{proof}

With this in hand, we may now proceed to prove the splitting result.

\begin{proof}[Proof of Theorem~\ref{singblowdownthm}]
Let~$E$ denote the subgraph of~$u$, as defined by~\eqref{Esubgraph}. We recall that, as observed right before the statement of Theorem~\ref{underPakmainthm}, the set~$E$ is~$\alpha$-minimal.

Let~$\C$ be a blow-down cone of~$E$. By definition, there exists a diverging sequence~${r_j}$ for which~$E_{r_j} = E/r_j \rightarrow \C$ in~$L^1_\loc(\R^{n + 1})$. As noticed in Lemma~\ref{blowdownlem},~$\C$ is a nontrivial~$\alpha$-minimal cone. Moreover,~$\C$ is not an half-space, since, otherwise,~$E$ would be a half-space too (again, by Lemma~\ref{blowdownlem}),
%and~$u$ would thus be affine, contradicting our hypotheses. Recall that this is equivalent to the cone~$\C$ being singular.
contradicting the hypothesis that~$E$ is the subgraph of a non-affine function. We also recall that this is equivalent to the cone~$\C$ being singular.

As~$E$ is a subgraph, it follows that~$E - t e_{n + 1} \subseteq E$ for every~$t > 0$. This yields that~$E_{r_j} - e_{n + 1} \subseteq E_{r_j}$ for every~$j$. Hence, by~$L^1_\loc(\R^{n + 1})$ convergence,~$\C - e_{n + 1} \subseteq \C$. Since~$\C$ is not a half-space, by Proposition~\ref{GMMprop} we conclude that~$\C$ is a cylinder in direction~$e_{n + 1}$, that is
\begin{equation} \label{Cen+1cyl}
\C + \lambda e_{n + 1} = \C \quad \mbox{for every } \lambda \in \R,
\end{equation}
or, equivalently,~$\C = \C' \times \R$, for some singular~$\alpha$-minimal cone~$\C' \subseteq \R^n$. Observe that the~$\alpha$-minimality of~$\C'$ is a consequence of~\cite[Theorem~10.1]{CRS10}. Also note that to obtain~\eqref{Cen+1cyl} we only took advantage of the fact that~$E$ is an~$\alpha$-minimal subgraph and not the hypotheses on the partial derivatives of~$u$.

Let now~$i = 1, \ldots, k$ be fixed. By the bound from below on the partial derivative~$\frac{\partial u}{\partial x_i}$ and the fundamental theorem of calculus, there exists a constant~$\kappa > 0$ such that
$$
u(z' + t e_i) - u(z') = \int_{0}^{t} \frac{\partial u(z' + \tau e_i)}{\partial x_i} \, d\tau \ge - \kappa t
$$
for every~$z' \in \R^n$ and~$t > 0$. Let now~$u_j$ be the function defining the blown-down set~$E_{r_j}$. Clearly,~$u_j(z') = u(r_j z') / r_j$ and hence
$$
u_j(y' + e_i) - u_j(y') = \frac{u(r_j y' + r_j e_i) - u(r_j y')}{r_j} \ge - \kappa
$$
for every~$y' \in \R^n$ and~$j \in \N$. This means that~$E_j - \kappa e_{n + 1} + e_i \subseteq E_j$ for every~$j \ge 1$. Passing to the limit and using~\eqref{Cen+1cyl}, we deduce that~$\C + e_i = \C - \kappa e_{n + 1} + e_i \subseteq \C$. Taking advantage once again of Proposition~\ref{GMMprop} and of the fact that~$\C$ is not a half-space, we infer that~$\C$ is a cylinder in direction~$e_i$ for every~$i = 1, \ldots, k$. The conclusion of Theorem~\ref{singblowdownthm} follows.
\end{proof}

\section{Proof of Theorem~\ref{underPakmainthm}} \label{mainsec}

\noindent
First of all, we may assume that the partial derivatives of~$u$ bounded on one side are the first~$n - \ell$. Also, up to flipping the variable~$x_i$, for some~$i \in \{ 1, \ldots, n - \ell \}$, we may suppose that those partial derivatives are all bounded from below. All in all, we have that
$$
\frac{\partial u}{\partial x_i} \ge - \kappa \quad \mbox{for every } i = 1, \ldots, n - \ell,
$$
for some constant~$\kappa \ge 0$.

If~$u$ were not affine, then, by applying Theorem~\ref{singblowdownthm} with~$k = n - \ell$, we would have that every blow-down cone~$\C$ of the set~$E$ defined by~\eqref{Esubgraph} is given by
$$
\C = \R^k \times P \times \R,
$$
for some singular~$\alpha$-minimal cone~$P \subseteq \R^{n - k} = \R^{\ell}$. As this contradicts assumption~\eqref{Paellprop}, we conclude that~$u$ must be affine.

\section{Proof of Theorem~\ref{nononflatthm}} \label{nononflatsec}

\noindent
Let~$\Pi$ be a half-space contained in~$E$. Without loss of generality, we may assume that~$\Pi = \{ x \in \R^n : x_{n + 1} < 0 \}$. Consider then a blow-down~$\C$ of~$E$, which is a nontrivial~$\alpha$-minimal cone, by Lemma~\ref{blowdownlem}. In particular, $\Pi \subseteq \C$ and~$0 \in \partial \Pi \cap \partial \C$. Using, e.g.,~\cite[Corollary~6.2]{CRS10}, we infer that~$\C = \Pi$ and therefore that~$E$ is half-space as well, thanks again to Lemma~\ref{blowdownlem}.

\section{Proof of Theorem~\ref{growthTeo}} \label{growthsec}

\noindent
Suppose by contradiction that the function~$u$ is not affine and denote with~$E$ its subgraph. Up to a translation of~$E$ in the vertical direction, hypothesis~\eqref{ugecone} yields that~$E$ contains the cone
$$
\mathscr{D} := \left\{ x \in \R^{n + 1} : x_{n + 1} < - C |x'| \right\}.
$$
Consider now a blow-down~$\C$ of~$E$. On the one hand, we clearly have that~$\mathscr{D} \subseteq \C$. On the other hand, by arguing as in the beginning of the proof of Theorem~\ref{singblowdownthm}, we have that~$\C$ must be a nontrivial vertical cylinder. More precisely,~$\C=\C'\times\R$, for some nontrivial singular~$\alpha$-minimal cone~$\C'\subseteq\R^n$. These two facts imply that~$\C' = \R^n$, contradicting its nontriviality. This concludes the proof.

\begin{remark} \label{stressedrmk}
By a refinement of this argument we can prove a stronger version of Theorem~\ref{growthTeo}, where hypothesis~\eqref{ugecone} is replaced by
\begin{equation} \label{ugecone+}
u(x') \ge - C(1 + |x'|) \quad \mbox{for every } x' \in \R^n \mbox{ such that } x_1 < 0.
\end{equation}
Indeed, arguing by contradiction as before, we see that any blow-down of the subgraph of~$u$ is a cylinder of the form~$\C' \times \R$. In light of~\eqref{ugecone+}, the cone~$\C'$ contains a half-space of~$\R^n$ and is thus flat, due to Theorem~\ref{nononflatthm}. This leads to a contradiction.
\end{remark}

%We denote
%\[
%\mathcal S:=\left\{(x',x_{n+1})\,:\,x_{n+1}<-c\big(1+|x'|\big)\right\},
%\]
%and we observe that the translated set $\mathcal S+c e_{n+1}$ is a cone centered at $0$. As a consequence,
%\[
%\mathcal S_{-c e_{n+1},r_j}=\frac{\mathcal S+c e_{n+1}}{r_j}=\mathcal S+c e_{n+1},
%\]
%for every sequnce $r_j\to\infty$, and hence $\mathcal S+c e_{n+1}$ is the unique blow-down of $\mathcal S$.
%
%Since $\mathcal S\subseteq E$, exploiting also Remark \ref{tang_cone_infty}, we obtain
%$\mathcal S+c e_{n+1}\subseteq\C'\times\R$. It is readily seen that this implies that $\C'=\R^n$, which contradicts the non triviality of $\C$,
%concluding the proof.

\section{Subgraphs of constant fractional mean curvature} \label{constapp}

\noindent
We pointed out in the introduction that if a function~$u:\R^n\to\R$ is regular enough in a neighborhood of a point~$x'\in\R^n$, then the quantity~$\cu_\alpha u(x')$ considered in~\eqref{Hcorsdef}-\eqref{Gdef} is well-defined.

In case~$u$ is merely a measurable function, we can still understand~$\cu_\alpha u$ as a linear form on the fractional Sobolev space~$W^{\alpha, 1}(\R^n)$, setting
\[
\langle\cu_\alpha u, v\rangle:=
\int_{\R^n}\int_{\R^n} G \! \left( \frac{u(x')-u(y')}{|x'-y'|} \right)
\! \left( v(x')-v(y') \right) \frac{dx' dy'}{|x'-y'|^{n+\alpha}}
\]
for every~$v\in W^{\alpha,1}(\R^n)$. This definition is indeed well-posed since~$G$ is bounded.

Let~$h$ be a real number. We say that a measurable function~$u:\R^n\to\R$ is a weak solution of~$\cu_\alpha u = h$ in~$\R^n$ if it holds
\begin{equation} \label{Halpha=hweak}
\langle\cu_\alpha u,v\rangle=h\int_{\R^n}v(x')\,dx' \quad \mbox{for every } v\in W^{\alpha,1}(\R^n).
\end{equation}
We remark that by the density of~$C^\infty_c(\R^n)$ in~$ W^{\alpha,1}(\R^n)$,
it is equivalent to consider the test functions~$v$ to be smooth and compactly supported.

We now prove that if the~$\alpha$-mean curvature of a global subgraph is constant, then this constant must be zero. More precisely, we have the following statement.

\begin{proposition} \label{c=0lem}
Let~$u:\R^n\to\R$ be a weak solution of~$\cu_\alpha u=h$ in~$\R^n$, for some constant~$h \in \R$. Then~$h=0$.
\end{proposition}

\begin{proof}
Recalling~\eqref{Gdef}, we notice that
\[
|G(t)|\le \int_0^{+\infty}\frac{d\tau}{(1+\tau^2)^\frac{n+1+\alpha}{2}}=:\Lambda<+\infty \quad \mbox{for every } t\in\R.
\]
Suppose that~$h\ge0$---the case~$h\le0$ is analogous. Let~$R>0$ and consider the test function~$v=\chi_{B'_R}\in W^{\alpha,1}(\R^n)$. We have
\[
|\langle\cu_\alpha u,\chi_{B'_R}\rangle|\le2\Lambda\int_{B'_R}\int_{\R^n\setminus B'_R}\frac{dx' dy'}{|x'-y'|^{n+\alpha}}
= C R^{n-\alpha},
\]
for some constant~$C>0$ depending only on~$n$ an~$\alpha$. Since~$u$ weakly solves~$\cu_\alpha = h$ in~$\R^n$, by plugging~$v = \chi_{B_R'}$ in~\eqref{Halpha=hweak} we deduce that
\[
h |B'_1| R^n=h\int_{\R^n}\chi_{B'_R}(x')\,dx'=\langle\cu_\alpha u,\chi_{B'_R}\rangle\le C R^{n-\alpha}
\]
for all~$R>0$, that is~$0 \le h R^\alpha \le C / |B_1'|$. Letting~$R\to+\infty$ we conclude that~$h=0$.
\end{proof}

We point out that, as a consequence of Proposition~\ref{c=0lem} and the results of~\cite{CL18},
%~\cite[Corollary~1.12]{CL18},
if a function~$u\in W^{\alpha, 1}_\loc(\R^n)$ is a weak solution of~$\cu_\alpha u = h$ in~$\R^n$, then the subgraph of~$u$ must be an~$\alpha$-minimal set---thus extending to the nonlocal framework a celebrated result of Chern, namely the Corollary of Theorem~1 in~\cite{C65}.
 
We further remark that other definitions for solutions of the equation~$\cu_\alpha u=h$ could have been considered, namely smooth
pointwise solutions and viscosity solutions---for a rigorous definition see~\cite[Subsection~4.3]{Lthesis}.
%~\cite[Definition 1.6]{CL18}).
However, it is readily seen that a smooth pointwise solution is also a viscosity solution. Moreover, in~\cite{CL18} it will be shown that
%~\cite[Corollary 1.8]{CL18} shows that
a viscosity solution is also a weak solution. Consequently, Proposition~\ref{c=0lem} applies to these other two notions of solutions as well.

\vfill


\begin{thebibliography}{99$\,$}

\bibitem{ADM11}
L. Ambrosio, G. De Philippis, L. Martinazzi,
\emph{Gamma-convergence of nonlocal perimeter functionals},
Manuscripta Math. \textbf{134} (2011), no. 3-4, 377--403.

\bibitem{BDM69}
E. Bombieri, E. De Giorgi, M. Miranda,
\emph{Una maggiorazione a priori relativa alle ipersuperfici minimali non parametriche},
Arch. Ration. Mech. Anal. \textbf{32} (1969), 255--267.

\bibitem{BLV16}
C. Bucur, L. Lombardini, E. Valdinoci,
\emph{Complete stickiness of nonlocal minimal surfaces for small values of the fractional parameter},
to appear in Ann. Inst. H. Poincar\'e Anal. Non Lin\'eaire, DOI: 10.1016/j.anihpc.2018.08.003, arXiv:1612.08295.

\bibitem{BV16}
C. Bucur, E. Valdinoci,
\emph{Nonlocal diffusion and applications},
Lecture Notes of the Unione Matematica Italiana, Vol. 20, Springer, Heidelberg, 2016.

\bibitem{CCS17}
X. Cabr\'e, E. Cinti, J. Serra,
\emph{Stable $s$-minimal cones in $\R^3$ are flat for $s \sim 1$}
preprint, arXiv:1710.08722.

\bibitem{CC17}
X. Cabr\'e, M. Cozzi,
\emph{A gradient estimate for nonlocal minimal graphs},
to appear in Duke Math.~J., arXiv:1711.08232.

\bibitem{CRS10}
L. A. Caffarelli, J.-M. Roquejoffre, O. Savin,
\emph{Nonlocal minimal surfaces},
Comm. Pure Appl. Math. \textbf{63} (2010), no. 9, 1111--1144.

\bibitem{CV13}
L. A. Caffarelli, E. Valdinoci,
\emph{Regularity properties of nonlocal minimal surfaces via limiting arguments},
Adv. Math. \textbf{248} (2013), 843--871.

\bibitem{C65}
S.-S. Chern,
\emph{On the curvatures of a piece of hypersurface in Euclidean space},
Abh. Math. Sem. Univ. Hamburg \textbf{29} (1965), 77--91.

\bibitem{CSV18}
E. Cinti, J. Serra, E. Valdinoci,
\emph{Quantitative flatness results and $BV$-estimates for stable nonlocal minimal surfaces},
to appear in J. Differential Geom.

\bibitem{CF17}
M. Cozzi, A. Figalli,
\emph{Regularity theory for local and nonlocal minimal surfaces: an overview},
Nonlocal and nonlinear diffusions and interactions: new methods and directions, 117--158, Lecture Notes in Math., Vol. 2186, Fond. CIME/CIME Found. Subser., Springer, Cham, 2017.

\bibitem{CL18}
M. Cozzi, L. Lombardini,
\emph{On nonlocal minimal graphs},
in preparation.

\bibitem{DPV12}
E. Di Nezza, G. Palatucci, E. Valdinoci,
\emph{Hitchhiker's guide to the fractional {S}obolev spaces},
Bull. Sci. Math. \textbf{136} (2012), no. 5, 521--573.

\bibitem{DSV18}
S. Dipierro, J. Serra, E. Valdinoci,
\emph{Improvement of flatness for nonlocal phase transitions},
to appear in Amer. J. Math., arXiv:1611.10105.

\bibitem{DV18}
S. Dipierro, E. Valdinoci,
\emph{Nonlocal minimal surfaces: interior regularity, quantitative estimates and boundary stickiness},
Recent Developments in Nonlocal Theory, 165--209, Partial Differential Equations and Measure Theory, De Gruyter Open, Berlin, 2018.

\bibitem{F15}
A. Farina,
\emph{A Bernstein-type result for the minimal surface equation},
Ann. Sc. Norm. Super. Pisa Cl. Sci. (5) \textbf{14} (2015), no. 4, 1231--1237.

\bibitem{F17}
A. Farina,
\emph{A sharp Bernstein-type theorem for entire minimal graphs},
Calc. Var. Partial Differential Equations \textbf{57} (2018), no. 5, 5 pp.

\bibitem{FarV17}
A. Farina, E. Valdinoci,
\emph{Flatness results for nonlocal minimal cones and subgraphs},
to appear in Ann. Sc. Norm. Super. Pisa Cl. Sci. (5), DOI: 10.2422/2036-2145.201708\_019, arXiv:1706.05701.

\bibitem{FV17}
A. Figalli, E. Valdinoci,
\emph{Regularity and {B}ernstein-type results for nonlocal minimal surfaces},
J. Reine Angew. Math. \textbf{729} (2017), 263--273.

\bibitem{G84}
E. Giusti,
\emph{Minimal surfaces and functions of bounded variation},
Monographs in Mathematics, Vol. 80, Birkh\"auser Verlag, Basel, 1984.

\bibitem{GMM97}
E. Gonzalez, U. Massari, M. Miranda,
\emph{On minimal cones}, 
Appl. Anal. \textbf{65} (1997), no. 1-2, 135--143.

\bibitem{Lthesis}
L. Lombardini,
\emph{Minimization problems involving nonlocal functionals: nonlocal minimal surfaces and a free boundary problem},
PhD thesis, 2018, available at arXiv:1811.09746.

\bibitem{Maggi}
F. Maggi,
\emph{Sets of finite perimeter and geometric variational problems},
Cambridge Stud. Adv.
Math. \textbf{135}, Cambridge Univ. Press (2012), Cambridge.

\bibitem{M64}
M. Miranda,
\emph{Distribuzioni aventi derivate misure insiemi di perimetro localmente finito},
Ann. Scuola Norm. Sup. Pisa (3) \textbf{18} (1964), 27--56. 

\bibitem{M61}
J. Moser,
\emph{On {H}arnack's theorem for elliptic differential equations},
Comm. Pure Appl. Math. \textbf{14} (1961), 577--591.

\bibitem{SV13}
O. Savin, E. Valdinoci,
\emph{Regularity of nonlocal minimal cones in dimension 2},
Calc. Var. Partial Differential Equations \textbf{48} (2013), no. 1-2, 33--39.

\bibitem{V13}
E. Valdinoci,
\emph{A fractional framework for perimeters and phase transitions},
Milan J. Math. \textbf{81} (2013), no. 1, 1--23. 

\bibitem{V91} A. Visintin,
\emph{Generalized coarea formula and fractal sets},
Japan J. Indust. Appl. Math. \textbf{8} (1991), no. 2, 175--201.

\end{thebibliography}
\end{document}